\documentclass[journal]{IEEEtran}

\ifCLASSINFOpdf
\else
\fi
\usepackage[cmex10]{amsmath}
\usepackage{paralist}
\usepackage{amsfonts}
\usepackage{amssymb}
\usepackage{array}
\usepackage{amsthm}
\newtheorem{theorem}{Theorem}
\newtheorem{lemma}[theorem]{Lemma}
\theoremstyle{definition}
\newtheorem{definition}{Definition}
\theoremstyle{remark}
\newtheorem*{remark}{Remark}
\usepackage{mathtools, cuted}
\usepackage{lipsum}
\newcommand{\E}{\mathbb{E}}
\newcommand{\ma}{\max\{(c(x,a)-z),0\}}

\newcommand{\armz}{{\underset{z \in \mathbb{R}}{\text{argmin}}}}
\newcommand{\lia}{{\underset{\Delta z \rightarrow 0}{\lim}}}
\newcommand{\sul}{{\overset{l-1}{\underset{i=1}{\sum}}}}
\newcommand{\sk}{{\overset{k-1}{\underset{i=1}{\sum}}}}
\newcommand{\skl}{{\overset{k-1}{\underset{i=l}{\sum}}}}
\newcommand{\skx}{{\overset{X}{\underset{i=k+1}{\sum}}}}
\newcommand{\slx}{{\overset{X}{\underset{i=l+1}{\sum}}}}
\newcommand{\cvr}{\text{CVaR}_{\alpha}(c(x_{k},a))}
\newcommand{\Prb}{\mathbb{P}}
\newcommand{\Prby}{\mathbb{P}_{y}}
\newcommand{\Prbyy}{\mathbb{P}_{y+1}}
\usepackage{comment}
\usepackage{url}

\usepackage{cite}
\usepackage{accents}
\usepackage{array}

\begin{document}

\title{Sequential Detection of Market shocks using Risk-averse Agent Based Models}
%
%
%
\author{Vikram Krishnamurthy, {\em Fellow, IEEE}  and Sujay Bhatt
\thanks{The authors are with the Department of Electrical and Computer Engineering at the University of British Columbia, Vancouver, V6T 1Z4, Canada.
email: {\tt vikramk@ece.ubc.ca}  and  {\tt sujaybhatt@ece.ubc.ca} }}
\maketitle

\begin{abstract}
This paper considers a  statistical signal processing problem involving agent based models of financial markets which at a micro-level are driven by socially aware and risk-averse trading agents. These agents trade (buy or sell) stocks by exploiting information about the decisions
of  previous agents (social learning) via an order book in addition to a private (noisy) signal they receive on the value of the stock. We are interested in the following: (1) Modelling the dynamics of these risk averse agents, (2) Sequential detection of a market shock based on the behaviour of these agents. 

Structural results which characterize social learning under a risk measure, CVaR (Conditional Value-at-risk), are presented and formulation of the Bayesian change point detection problem is provided. The structural results exhibit two interesting properties: (i) Risk averse agents herd more often than risk neutral agents (ii) The stopping set in the sequential detection problem is non-convex. The framework is validated on data from the Yahoo! Tech Buzz game dataset and it is revealed that (a) The model identifies the value changes based on agent's trading decisions. (b) Reasonable quickest detection performance is achieved when the agents are risk-averse.  
\end{abstract}

\begin{IEEEkeywords}
conditional value at risk (CVaR), social learning filter, market shock, quickest detection, agent based models, monotone Bayesian update
\end{IEEEkeywords}

%

\section{Introduction} \label{sec:intro}
%
%
%
%

Financial markets evolve based on the behaviour of a large number of interacting entities. Understanding the interaction of these agents is therefore essential in statistical inference from financial data. This motivates the study of ``agent based models'' for financial markets. 
Agent based models  are  useful for capturing the global behaviour of highly interconnected financial systems by simulating the behaviour of the local interacting systems \cite{Leb06}, \cite{Leb00}, \cite{SZSL07}, \cite{ACPZ09}. Unlike standard economic models which emphasize the equilibrium properties of financial markets, agent based models stress local interactions and out-of-equilibrium dynamics that may not reach equilibrium in the long run \cite{TJ06}. Agent based models  are commonly used to determine the conditions that lead a group of interacting agents to form an aggregate behaviour \cite{CB00}, \cite{AZ98}, \cite{PS11}, \cite{Cha04} and to model stylized facts like correlation of returns and volatility clustering \cite{Con07}, \cite{LM00}. Agent based models  have also been used model anomalies that the standard approaches fail to explain like ``fat tails", absence of simple arbitrage, gain/loss asymmetry and leverage effects \cite{CMZ13}, \cite{CPZ11}. 

In this paper, we are interested in developing agent based models for studying global events in financial markets where the underlying value of the stock experiences a jump change (shock). Market shocks are known to affect stock market returns \cite{AM09}, cause fluctuations in the economy \cite{GYZ09} and necessitate market making \cite{DM09}. Therefore detecting shocks is essential and when the interacting agents are acting based on private signals and complete history of other agents' trading decisions, it is non-trivial \cite{Kri12}.    

  The problem of market shock detection in the presence of social learning considered in this paper is different from a standard signal processing (SP) problem in the following ways:
\begin{compactenum}
\item Agents (or social sensors) influence the behaviour of other agents, whereas in standard SP sensors typically do not affect other sensors.
\item Agents  reveal quantized information (decisions) and have dynamics, whereas in standard SP sensors are static with the dynamics modelled
in the state equation.
\item Standard SP is expectation centric. In this paper  we use  {\em coherent risk measures} which generalizes
the concept of expected value and is much more relevant in financial applications. Such coherent risk measures \cite{ADEH02} are now widely used in finance to model risk averse behaviour.
\end{compactenum}

Properties 1 and 2 above are captured by social learning models. Such social learning models,
 where agents face \textit{fixed} prices, are  considered in \cite{BHW92}, \cite{Wel92}, \cite{Ban92}, \cite{Cha04}. They show that after a finite amount of time,  an informational cascade takes place and all subsequent agents choose the same action regardless of their private signal. Models where agents act sequentially to optimize local costs (to choose an action) and are socially aware were considered in \cite{AZ98}, \cite{Glo89}. This paper considers a similar model, but, 
in order to incorporate property 3 above (risk averse behaviour), we will replace the classical
social learning model of expected cost minimizers to that of risk averse minimizers.
  The resulting risk-averse social learning filter has several interesting (and unusual) properties that will
 be discussed in the paper.

\subsection*{Main Results and Organization}
 Section \ref{sec:Model} presents the social learning agent based model and the market observer's objective for detecting shocks. The formulation involves the interaction
 of local and global decision makers.  Individual agents perform social learning and the market maker seeks to determine if the underlying asset value
 has changed based on the agent behaviour. The shock in the asset value changes at a phase distributed time (which generalizes
 geometric distributed change times).  The problem of market shock detection considered in this paper is different from the classical Bayesian quickest detection \cite{SA07}, \cite{PH08}, \cite{Fri09} where, local observations are used to detect the change. Quickest detection in the presence of social learning was considered in \cite{Kri12} where it was shown that making global decisions (stop or continue) based on local decisions (buy or sell) leads to discontinuous value function and the optimal policy has multiple thresholds. However, unlike \cite{Kri12} which deals with expected cost, we consider a more general measure to account for the local agents' attitude towards risk.  

It is well documented in various fields like economics \cite{CLLS75}, behavioural economics, psychology \cite{DS99} that people prefer a certain but possibly less desirable outcome over an uncertain but potentially larger outcome. To model this risk averse behaviour, commonly used risk measures{\footnote{A risk measure $\varrho : \mathcal{L} \rightarrow \mathbb{R}$ is a mapping from the space of measurable functions to the real line which satisfies the following properties: (i) $\varrho(0)=0$. (ii) If $S_{1}, S_{2} \in \mathcal{L}$ and $S_{1} \leq S_{2} ~\text{a.s}$ then $\varrho(S_{1}) \leq \varrho(S_{2})$. (iii) if $a\in\mathbb{R}$ and $S\in\mathcal{L}$, then $\varrho(S+a) = \varrho(S)+a $. The risk measure is coherent if in addition $\varrho$ satisfies: (iv) If $S_{1}, S_{2} \in \mathcal{L}$, then $\varrho(S_{1}+S_{2}) \leq \varrho(S_{1}) + \varrho(S_{1})$. (v) If $a \geq 0$ and $S\in\mathcal{L}$, then $\varrho(aS)=a\varrho(S)$. The expectation operator is a special case where subadditivity is replaced by additivity.}} are Value-at-Risk (VaR), Conditional Value-at-Risk (CVaR), Entropic risk measure and Tail value at risk; see \cite{MJ10}. We consider social learning under CVaR risk measure. CVaR \cite{RU00} is an extension of VaR  that gives the total loss given a loss event and is a coherent risk measure \cite{ADEH02}. In this paper, we choose CVaR risk measure as it exhibits the following properties \cite{ADEH02}, \cite{RU00}:
(i) It associates higher risk with higher cost.
(ii) It ensures that risk is not a function of the quantity purchased, but arises from the stock.
(iii) It is convex. 
CVaR as a risk measure has been used in solving portfolio optimization problems \cite{PUK99}, \cite{LSU10} credit risk optimization \cite{AMRU01}, order execution \cite{FRP12} and also to optimize an insurer's product mix \cite{TWT10}. For an overview of risk measures and their application in finance, see \cite{MJ10}.  

Section \ref{sec:prop} provides structural results which characterize the social learning under CVaR risk measure and its properties.
We show that, under reasonable assumptions on the costs, the trading decisions taken by socially aware and risk-averse agents are ordinal functions of their private observations and monotone in the prior information. This implies that the Bayesian social learning follows simple intuitive rules. The change point detection problem is formulated as a Market Observer{\footnote{The market observer could be the securities dealer (investment bank or syndicate) that underwrites the stock which is later traded in a secondary market.}} seeking to detect a shock in the stock value (modelled as a Markov chain) by balancing the natural trade-off between detection delay and false alarm.
 
 Section  \ref{sec:disc} discusses the unusual properties exhibited by the CVaR social learning filter and explores the link between local and global behaviour in 
agent based models  for detection of market shocks. 
   We show that the stopping region
 for the sequential detection problem is non-convex; this is in contrast to  standard signal processing quickest detection problems where the stopping set is convex. Similar results were developed in \cite{Kri11,Kri12,KP14}.
 
 Finally, Section  \ref{sec:dat} discusses an application of the agent based model and change detection framework in a stock market data set.
 We use a data set  from Tech Buzz Game which is a stock market simulation launched by Yahoo! Research and O'Reilly Media 
 to gain insights into forecasting high-tech events and trades. Tech Buzz  uses Dynamic parimutuel markets (DPM) as its trading mechanism. DPMs are known to provide accurate predictions in field studies on  price formation in election stock markets \cite{FRR99}, mechanism design for sales forecasting \cite{PC02} and betting in sports markets \cite{TZ88}, \cite{GDBZ98}.

\section{CVaR Social Learning Model and Market Observer's objective} \label{sec:Model}
 This section presents the Bayesian social learning model and defines the objective of the market observer. As will be shown later
 in Section \ref{sec:prop}, the model results in ordinal decision making thereby mimicking human behavior
and the risk measure captures a trader's attitude towards risk.

\subsection{CVaR Social Learning Model}
 The market micro-structure is modelled as a discrete time dealer market motivated by algorithmic and high-frequency tick-by-tick trading \cite{CJ13}. There is a single traded stock or asset, a market observer and a countable number of trading agents. The asset has an initial true underlying value $x_{0} \in \mathcal{X} = \lbrace 1,2,\hdots, X \rbrace$. The market observer does not receive direct information about $x\in \mathcal{X}$ but only observes the public buy/sell actions of agents, $a_{k} \in \mathcal{A} = \lbrace 1(\text{buy}),2({\text{sell}}) \rbrace$. The agents themselves receive noisy private observations of the underlying value $x$ and consider this in addition to the trading decisions of the other agents visible in the order book \cite{AT15}, \cite{AS08}, \cite{KA12,KA13}. At a random time, $\tau^{0}$ determined by  the transition matrix $P$, the asset experiences a jump change in its value to a new value. The aim of the market observer is to detect the change time (global decision) with minimal cost, having access to only the actions of these socially aware agents. Let $y_{k} \in \mathcal{Y} = \{1,2, \hdots, Y\}$ denote agent $k$'s private observation. The initial distribution is $\pi_{0} = (\pi_{0}(i),i\in \mathcal{X})$ where $\pi_{0}(i) = \Prb(x_{0}=i)$. \\

The  agent based model has the following dynamics:
\begin{compactenum}
\item[1.] \textit{Shock in the asset value}: At time $\tau^{0} > 0$, the asset experiences a jump change (shock) in its value due to exogenous factors. 
The change point  $\tau^0$ is modelled by a   
{\em  phase type (PH) distribution}. 
The family of all PH-distributions forms a dense subset for the set of all distributions
	\cite{Neu89} i.e., for any given distribution function $F$ such that $F(0) = 0$, one can find a sequence of PH-distributions	
$\{F_n , n	\geq	1\}$	 to		approximate	$F$	uniformly over $[0, \infty)$.
The PH-distributed time $\tau^0$ can be constructed via
 a multi-state  Markov chain $x_k$ with state space $\mathcal{X} = \{1,\hdots, X\}$ as follows:
Assume    state `1'   is an absorbing state
and denotes the state after the jump change.    The states $2,\ldots,X$ (corresponding to beliefs $e_2,\ldots,e_X$) can be viewed as  a single composite state that $x$ resides in before the jump. 
 So $\tau^{0} = \text{inf} \lbrace{k:x_{k}=1} \rbrace$ and the transition probability matrix $P$ is of the form
\begin{equation}
P = \begin{bmatrix}
       1 & 0        \\
       \underline{P}_{(X-1)\times1}   & {\bar{P}_{(X-1)\times(X-1)}}
     \end{bmatrix}
\end{equation}
 The 
distribution of the absorption time to state 1 is 
\begin{equation} \label{eq:nu}
 \nu_0 = \pi_0(1), \quad \nu_k = \bar{\pi}_0' \bar{P}^{k-1} \underline{P}, \quad k\geq 1, \end{equation}
 where $\bar{\pi}_0 = [\pi_0(2),\ldots,\pi_0(X)]'$.
The key idea  is that by appropriately choosing the pair $(\pi_0,P)$ 
and the associated state space dimension $X$,
 one can approximate any given discrete distribution on $[0, \infty)$ by the distribution $\{\nu_k, k \geq 0\}$; see 
\cite[pp.240-243]{Neu89}.
 The event $\{x_k = 1\}$ means the change point has occurred at time $k$ according
 to PH-distribution (\ref{eq:nu}). In the special case when $x$ is a 2-state Markov chain,
 the change time $\tau^0$  is geometrically distributed.

\item[2.] \textit{Agent's Private Observation}: Agent $k$'s private (local) observation denoted by $y_{k}$ is a noisy measurement of the true value of the asset. It is obtained from the observation likelihood distribution as,
\begin{equation}
B_{xy} = \Prb(y_{k}=y|x_{k}=x)
\end{equation} 

\item[3.] \textit{Private Belief update}: Agent $k$ updates its private belief using the observation $y_{k}$ and the prior public belief $\pi_{k-1}(i) = \Prb(X=i|a_{1},\hdots,a_{k-1})$ as the following Hidden Markov Model update
\begin{equation}
\eta_{k} = \frac{B_{y_{k}}P'\pi_{k-1}}{\textbf{1}'B_{y_{k}}P'\pi_{k-1}}
\end{equation}
\item[4.] \textit{Agent's trading decision}: Agent $k$ executes an action $a_{k}\in\mathcal{A}=\lbrace1(\text{buy}),2(\text{sell})\rbrace$ to myopically minimize its cost. Let $c(i,a)$ denote the cost incurred if the agent takes action $a$ when the underlying state is $i$. Let the local cost vector be 
\begin{equation}
c_{a} = [c(1,a)~c(2,a) \dots ~c(X,a)]
\end{equation}
The costs for different actions are taken as
\begin{equation}
c(i,j) = p_{j}-\beta_{ij} ~ \text{for} ~ i \in \mathcal{X}, j \in \mathcal{A}
\end{equation}
where  $\beta_{ij}$ corresponds to the agent's demand. Here demand is the agent's desire and willingness to trade at a price $p_{j}$ for the stock. Here $p_{1}$ is the quoted price for purchase and $p_{2}$ is the price demanded in exchange for the stock. We assume that the price is the same during the period in which the value changes. As a result, the willingness of each agent only depends on the degree of uncertainty on the value of the stock.
\begin{remark}
The analysis provided in this paper straightforwardly extends  to the case when different agents are facing different prices like in an order book \cite{AT15}, \cite{AS08}, \cite{KA13}. For notational simplicity we assume the cost are time invariant.
\end{remark}

  The agent considers measures of risk in the presence of uncertainty in order to overcome the losses incurred in trading. To illustrate this, let $c(x,a)$ denote the loss incurred with action $a$ while at unknown and random state $x\in{\mathcal{X}}$. When an agent solves an optimization problem involving $c(x,a)$ for selecting the best trading decision, it will take into account not just the expected loss, but also the ``riskiness" associated with the trading decision $a$. 
The agent therefore chooses an action $a_{k}$ to minimize the CVaR   measure\footnote{
For the reader unfamiliar with risk measures, it should be noted that CVaR is one of the `big' developments in risk modelling in finance in the last 15 years.
In comparison, the 
 value at risk (VaR) is the percentile loss namely,  $\text{VaR}_\alpha(x) = \min\{z: F_x(z) \geq \alpha\} $ for cdf $F_x$. While CVaR is a coherent risk measure,
 VaR is not  convex and so not coherent. CVaR has  other remarkable properties \cite{RU00}: it is continuous in $\alpha$
 and jointly convex in $(x,\alpha)$.
For continuous cdf $F_x$, $\text{CVaR}_\alpha(x) = \E\{X | X > \text{VaR}_\alpha(x)\}$. Note that the variance  is
 not  a coherent risk~measure.} of trading as
\begin{align}
a_{k} &=  {\underset{a \in \mathcal{A}}{\text{argmin}}} \{ \cvr \} \\
&=  {\underset{a \in \mathcal{A}}{\text{argmin}}} \{ {\underset{z \in \mathbb{R}}{\text{min}}} ~ \{ z + \frac{1}{\alpha} \mathbb{E}_{y_{k}}[{\max} \{ (c(x_{k},a)-z),0 \rbrace] \} \} \nonumber
\end{align}
Here $\alpha \in (0,1]$ reflects the degree of risk-aversion for the agent (the smaller $\alpha$ is, the more risk-averse the agent is). 
Define 
\begin{equation}
\mathcal{H}_{k} := \sigma \text{- algebra generated by}~ (a_{1},a_{2},\hdots,a_{k-1},y_{k}) 
\end{equation} 
$\mathbb{E}_{y_{k}}$ denotes the expectation with respect to private belief, i.e, $\mathbb{E}_{y_{k}} = \mathbb{E}[.|\mathcal{H}_{k}]$ when the private belief is updated after observation $y_{k}$.
 

\item[5.] \textit{Social Learning and Public belief update}: Agent $k$'s action is recorded in the order book and hence broadcast publicly. Subsequent agents and the market observer update the public belief on the value of the stock according to the social learning Bayesian filter as follows 
\begin{equation}
\pi_{k} = T^{\pi_{k-1}} (\pi_{k-1},a_k)  = \frac{R_{a_{k}}^{\pi_{k-1}}P'\pi_{k-1}}{\textbf{1}'R_{a_{k}}^{\pi_{k-1}}P'\pi_{k-1}}
\end{equation}

Here, $R_{a_{k}}^{\pi_{k-1}} = \text{diag}(\Prb(a_{k}|x=i,\pi_{k-1}),i \in \mathcal{X})$, where $\Prb(a_{k}|x=i,\pi_{k-1}) = {\underset{y \in \mathcal{Y}}{\sum}} \Prb(a_{k}|y,\pi_{k-1})\Prb(y|x_{k}=i)$ and 
\[ \Prb(a_{k}|y,\pi_{k-1}) = \left\{ \begin{array}{ll}
         1 & \mbox{if $ a_{k} = {\underset{a \in \mathcal{A}}{\text{argmin}}}~ \text{CVaR}_{\zeta}(c(x_{k},a))$} ; \\
         0 & \mbox{$\text{otherwise}$}.\end{array} \right. \] 
Note that $\pi_k$ belongs to the unit simplex  $\Pi(X){\overset{\Delta}{=}}\lbrace \pi \in \mathbb{R}^{X} : \textbf{1}_{X}'\pi = 1, 0 \leq \pi \leq 1 ~\text{for all} ~ i \in \mathcal{X}\rbrace $. 
\item[6.]\textit{ Market Observer's Action}: The market observer (securities dealer) seeks to achieve quickest detection by balancing delay with false alarm. At each time $k$, the market observer chooses action\footnote{It is important to distinguish between the ``local'' decisions $a_k$ of the agents and ``global'' decisions $u_k$ of the market maker. Clearly the decisions 
$a_k$ affect the choice of $u_k$ as will be made precise below.} $u_{k}$ as
\begin{equation}
u_{k} \in \mathcal{U} =  \lbrace 1(\text{stop}), 2(\text{continue}) \rbrace
\end{equation}
Here `Stop' indicates that the value has changed and the dealer incorporates this information before selling new issues to investors. The formulation presented considers a general parametrization of the costs associated with detection delay and false alarm costs. 
Define  
\begin{equation}
\mathcal{G}_{k} := \sigma \text{- algebra generated by}~  (a_{1},a_{2},\hdots,a_{k-1},a_{k}).
\end{equation}  
\begin{compactitem}
\item[i)] \textit{Cost of Stopping}: The asset experiences a jump change(shock) in its value at time $\tau^{0}$. If the action $u_{k}=1$ is chosen before the change point, a false alarm penalty is incurred. This corresponds to the event ${\underset{i\geq2}\cup} \lbrace x_{k}=i \rbrace \cap \lbrace u_{k}=1 \rbrace$. Let $f_{i}\mathcal{I}(x_{k}=i,u_{k}=1)$ denote the cost of false alarm in state $i,i\in \mathcal{X}$ with $f_{i}\geq0$. The expected false alarm penalty is
\begin{align} \label{eq:falc}
C(\pi_{k},u_{k}=1) &= {\underset{i\in\mathcal{X}}{\sum}}f_{i}\mathbb{E}\lbrace\mathcal{I}(x_{k}=i,u_{k}=1)|\mathcal{G}_{k} \rbrace \nonumber \\
 &= \textbf{f}'\pi_{k}
\end{align}
where $\textbf{f} = (f_{1},\hdots,f_{X})$ and it is chosen with increasing elements, so that states further from `$1$' incur higher penalties. Clearly, $f_{1}=0$.

\item[ii)] \textit{Cost of delay}: A delay cost is incurred when the event $\lbrace x_{k}=1,u_{k}=2 \rbrace$ occurs, i.e, even though the state changed at $k$, the market observer fails to identify the change. The expected delay cost is 
\begin{align} \label{eq:delc}
C(\pi_{k},u_{k}=2) &= d\, \mathbb{E}\lbrace\mathcal{I}(x_{k}=i,u_{k}=1)|\mathcal{G}_{k} \rbrace \nonumber \\
&= de_{1}'\pi_{k}
\end{align}
where $d>0$ is the delay cost and $e_{1}$ denotes the unit vector with 1 in the first position.
\end{compactitem}
\end{compactenum}

\subsection{Market Observer's Quickest Detection Objective}
The market maker chooses its action at each time $k$ as 

\begin{equation} u_k = \mu(\pi_k)  \in  \lbrace 1(\text{stop}), 2(\text{continue}) \rbrace \label{eq:marketaction}
\end{equation}
 where $\mu$ denotes a stationary policy.
For each initial distribution $\pi_{0} \in \Pi(X)$ and policy $\mu$, the following cost is associated
\begin{equation} \label{eq:Obj}
J_{\mu}(\pi_{0}) = \mathbb{E}^{\mu}_{\pi_{0}} \left\{ {\overset{\tau-1} {\underset{k=1} {\sum}}} \rho^{k-1}C(\pi_{k},u_{k}=2)+\rho^{\tau-1}C(\pi_{k},u_{k}=1) \right\}
\end{equation}
where $\rho \in [0,1]$ denotes an economic  discount factor. 
(As long as $\textbf{f} $ is non-zero, stopping is guaranteed in finite time and so $\rho = 1$ is allowed.)

Given the cost, the market observer's objective is to determine $\tau^{0}$ with minimum cost by computing an optimal policy $\mu^{*}$ such that
\begin{equation}  \label{eq:pomdp}
J_{\mu^{*}}(\pi_{0}) = {\underset{\mu\in \boldsymbol{\mu}}{\text{inf}}J_{\mu}(\pi_{0})}
\end{equation}
The sequential detection problem 
(\ref{eq:pomdp}) can be viewed as a  partially observed Markov decision process (POMDP) where  the belief update is given by the social
learning filter.

\subsection{Stochastic Dynamic Programming Formulation} \label{sec:sdp}

The optimal policy of the market observer $\mu^{*}:\Pi(X)\rightarrow\{1,2\}$ is the solution of \eqref{eq:Obj} and is given by Bellman's dynamic programming equation as follows: 
\begin{align}
V(\pi) &= \text{min} \left\{ C(\pi,1),C(\pi,2)+\rho{\underset{a\in \mathcal{A}}{\sum}}V(T^{\pi} (\pi,a))\sigma(\pi,a)\right\} \label{eq:val}\\
\mu^{*}(\pi) &= \text{argmin} \left\{ C(\pi,1),C(\pi,2)+\rho{\underset{a\in \mathcal{A}}{\sum}}V(T^{\pi}(\pi,a))\sigma(\pi,a) \nonumber \right\}
\end{align}
where $T^{\pi} (\pi,a) = \frac{R_{a}^{\pi}P'\pi}{\textbf{1}'R_{a}^{\pi}P'\pi}$ is the CVaR-social learning filter.  $C(\pi,1)$ and $C(\pi,2)$ from \eqref{eq:falc} and \eqref{eq:delc} are the market observer's costs. Here $\rho \in [0,1]$ is the discount factor which is a measure of the degree of impatience of the market observer. As $C(\pi,1)$ and $C(\pi,2)$ are non-negative and bounded for $\pi \in \Pi(X)$, the stopping time $\tau$ is finite for all $\rho \in [0,1]$.

The aim of the market observer is then to determine the stopping set $\mathcal{S} = \{ \pi \in \Pi(X) : \mu^{*}(\pi) = 1 \}$ given by:
\begin{equation*}
\mathcal{S} = \left\{ \pi : C(\pi,1) < C(\pi,2) + \rho {\underset{a\in \mathcal{A}}{\sum}}V(T^{\pi}(\pi,a))\sigma(\pi,a) \right\}
\end{equation*}

\section{Properties of CVar Social Learning Filter} \label{sec:prop}

 This section discusses the main results regarding the structural properties of the CVaR social learning filter and highlights the significant role it plays in charactering the properties of market observer's value function and optimal policy.  According to Theorem  \ref{thm:a}, risk-averse agents take decisions that are monotone and ordinal in the observations and monotone in the prior; and its monotone ordinal behaviour implies that a Bayesian model chosen in this paper is a useful idealization.  

\subsection{Assumptions}
 
The following assumptions will be used throughout the paper:
\begin{enumerate}
\item[(A1)] Observation matrix $B$ and transition matrix $P$ are TP2 (all second order minors are non-negative)
\item[(A2)] Agents' local cost vector $c_{a}$ is sub-modular. That is $c(x,2)-c(x,1)\leq c(x+1,2)-c(x+1,1)$.
\end{enumerate} 
The matrices being TP2 \cite{KR80} ensures that the public belief Bayesian update can be compared with the prior \cite{Lov87} and sub-modular \cite{Top98} costs ensure that if it is less risky to choose  $a = 2$ when in $x$, it is also less risky to choose it when in $x+1$.
\subsection{Properties of CVaR social learning filter}

The $\mathcal{Y}\times \mathcal{A}$ local decision likelihood probability matrix $R^{\pi}$ (analogous to observation likelihood) can be computed as
\begin{align} \label{eq:rpi}
R^{\pi} &= BM^{\pi}, \text{where}~ M^{\pi}_{y,a} {\overset{\Delta}{=}} \Prb(a|y,\pi) \\
\Prb(a|y,\pi) &= \mathcal{I}(\text{CVaR}_{\alpha}(c(x_{k},a)) < \text{CVaR}_{\alpha}(c(x_{k},a'))) \nonumber
\end{align}
 where $a' = A/\{a\}$. Here $\mathcal{I}$ denotes the indicator function. 
 
Let $H^{\alpha}(y,a)= \cvr$ denote the cost with CVaR measure, associated with action $a$ and observation $y$ for convenience i.e, 
\begin{equation} \label{eq:hya}
H^{\alpha}(y,a) =   {\underset{z \in \mathbb{R}}{\text{min}}} ~ \{ z + \frac{1}{\alpha} \mathbb{E}_{y}[{\max} \{ (c(x,a)-z),0 \rbrace] \} 
\end{equation}

Here $\mathbb{E}_{y} = \mathbb{E}[.|\mathcal{H}_{k}]$. $y$ indicates the dependence of  $\mathbb{E}$ and hence $H^{\alpha}$ on the observation.  Let $a^{*}(\pi,y) = {\text{argmin}} ~ H^{\alpha}(y,a)$ denote the optimal action of the agent with explicit dependence on the distribution and observation.

The following result says that agents choose a trading decision that is monotone and ordinal in their private observation. Humans typically convert numerical attributes to ordinal scales before making  decisions. For example,
it does not matter if the cost of a meal at a restaurant is \$200 or \$205; an individual would classify this cost as ``high". 
Also credit rating agencies use ordinal symbols such as AAA, AA, A.

\begin{theorem}\label{thm:a}
Under (A1) and (A2), the action $a^{*}(\pi,y)$ made by each agent is increasing and hence ordinal in $y$ for any prior belief $\pi$.\\
Under (A2), $a^{*}(\pi,y)$  is increasing in $\pi$ with respect to the monotone likelihood ratio order (Definition \ref{def:mlrdef} in the appendix).
\end{theorem}
The proof is given in the appendix. Theorem \ref{thm:a} says that agents exhibit monotone ordinal behaviour. The condition that $a^{*}(\pi,y)$ is monotone in the observation $y$ is required to characterize the local decision matrices on different regions in the belief space which is stated next. 


\begin{theorem} \label{thm:p}
Under (A1) and (A2), there are at most $Y+1$ distinct local decision likelihood matrices $R^{\pi}$ and the belief space $\Pi(X)$ can be partitioned into the following $Y+1$ polytopes:
\begin{IEEEeqnarray}{c}
\mathcal{P}^{\alpha}_{1} = \{\pi\in \Pi(X) : H(1,1)-H(1,2) \geq 0 \} \IEEEnonumber\\
\quad \mathcal{P}^{\alpha}_{l} = \{\pi\in \Pi(X) : H(l-1,1)-H(l-1,2) < 0 \\ 
\quad ~ \cap ~ H(l,1)-H(l,2) \geq 0 \}, ~ l=2,\hdots,Y \IEEEnonumber \\
\mathcal{P}^{\alpha}_{Y+1} = \{\pi\in \Pi(X) : H(Y,1)-H(Y,2) < 0 \} \IEEEnonumber
\end{IEEEeqnarray}
Also, the matrices $R^{\pi}$ are constant on each of these polytopes.
\end{theorem}

The proof is given in the appendix. Theorem \ref{thm:p} is required to specify the policy for the market observer. Indeed it leads to unusual
behavior (non-convex) stopping regions in quickest detection as described in Section \ref{sec:nonconvex}.

\section{Social Learning and Change Detection  for risk-averse agents} \label{sec:disc}

This section illustrates the properties of the risk-averse social learning filter which leads to a non-convex value function and therefore non-convex
stopping set of quickest detection.
 
 \subsection{Social Learning Behavior of Risk Averse Agents}
The following discussion highlights the relation between risk-aversion factor $\alpha$ and the regions $\mathcal{P}^{\alpha}_{l}$. For a given risk-aversion factor $\alpha$, Theorem \ref{thm:p} shows that there are at most $Y+1$ polytopes on the belief space. It was shown in \cite{Kri12} that for the risk neutral case with $X=2$, and $P = I$ (the value is a random variable) the intervals $\mathcal{P}^{\alpha}_{1}$ and $\mathcal{P}^{\alpha}_{3}$ correspond to the herding region and the interval $\mathcal{P}^{\alpha}_{2}$ corresponds to the social learning region. In the herding region, the agents take the same action as the belief is frozen. In the social learning region there is observational learning. However, when the agents are optimizing a more general risk measure (CVaR), the social learning region is different for different risk-aversion factors. The social learning region for the CVaR risk measure is shown in Fig. \ref{alp_p}.
\begin{figure}[!t] 
\centering
\includegraphics[scale=0.4]{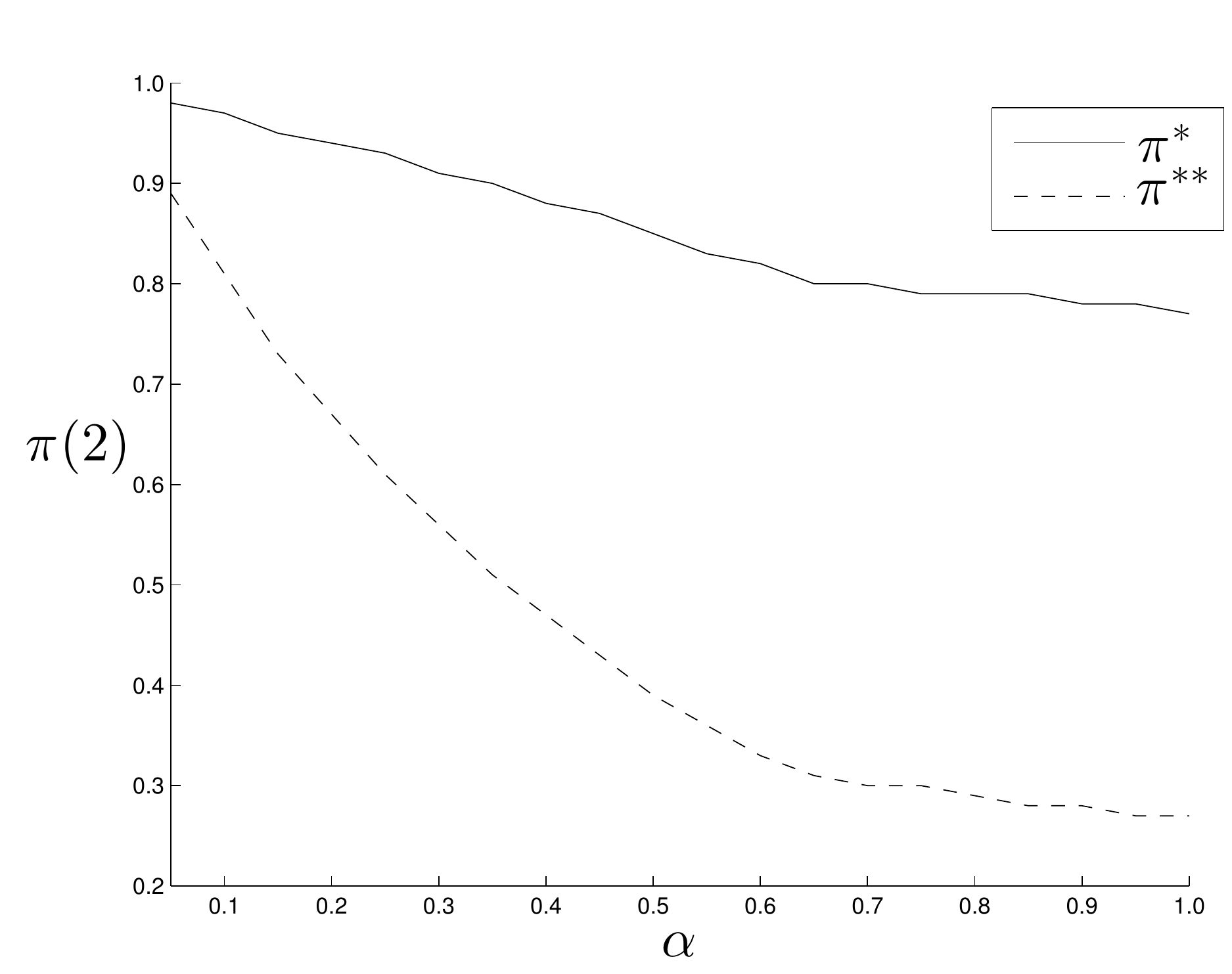}
\caption{The social learning region for $\alpha \in (0,1]$. It can be seen that the curves corresponding to $\pi^{**}$ and $\pi^{*}$ do not intersect and their separation (social learning region) varies with $\alpha$. Here $P=I$, i.e, the value is a random variable. 
 }
\label{alp_p}
\end{figure}
The following parameters were chosen: 
\[ B = \begin{bmatrix}
0.8 & 0.2 \\
0.3 & 0.7
\end{bmatrix},
P =\begin{bmatrix}
1 & 0 \\
0 & 1
\end{bmatrix},
c =\begin{bmatrix}
1 & 2 \\
3 & 0.5
\end{bmatrix}.
\]
It can be observed that the width of the social learning region decreases as $\alpha$ decreases. This can be interpreted as risk-averse agents showing a larger tendency to go with the crowd rather than ``risk" choosing the other action. With the same $B$ and $c$ parameters, but with 
transition matrix
\[
P = \begin{bmatrix}
1 & 0 \\
0.1 & 0.9
\end{bmatrix}
\]
the social learning region is shown in Fig. \ref{alp_p2}.
\begin{figure}[!t] 
\centering
\includegraphics[scale=0.4]{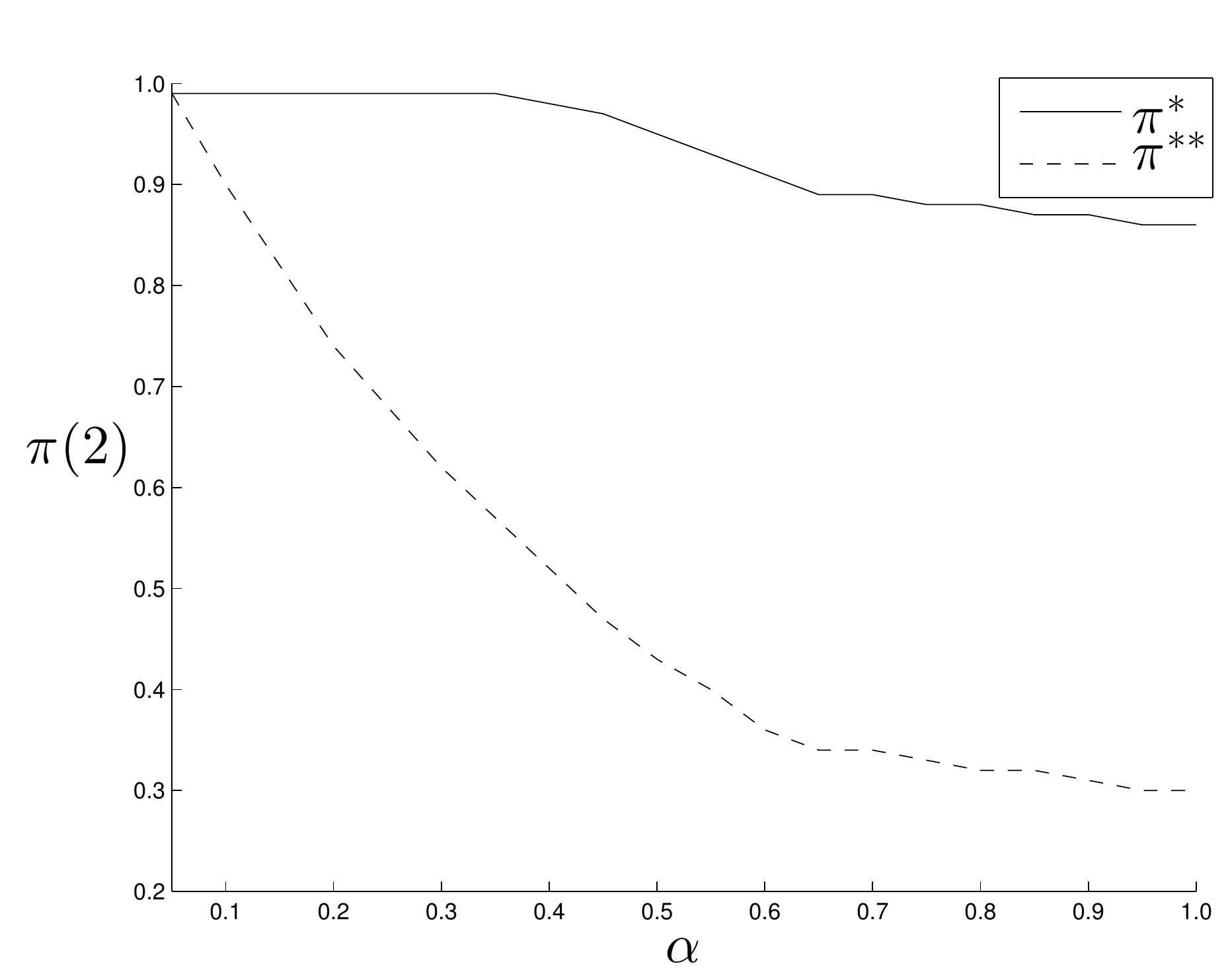}
\caption{The social learning region for $\alpha \in (0,1]$. It can be seen that the social learning region is absent when agents are sufficiently risk-averse and is larger when the stock value is known to change, i.e, $P\neq I$.
 }
\label{alp_p2}
\end{figure}
From Fig. \ref{alp_p2}, it is observed that when the state is evolving and when the agents are sufficiently risk-averse, social learning region is very small. It can be interpreted as: agents having a strong risk-averse attitude don't prefer to ``learn" from the crowd; but rather face the same consequences, when $P \neq I$.

\subsection{Nonconvex   Stopping Set for  Market Shock Detection} \label{sec:nonconvex}
We now illustrate the solution to the Bellman's stochastic dynamic programming equation (\ref{eq:val}), which determines the optimal policy for 
quickest market shock detection, by considering an agent based model with two states. Clearly the agents (local decision makers) and market observer  interact -- the  local decisions $a_k$ taken by the agents determines the public belief
   $\pi_k$ and hence determines decision $u_{k}$ of the  market observer via (\ref{eq:marketaction}).

 From Theorem \ref{thm:p}, the polytopes $\mathcal{P}^{\alpha}_{1}, \mathcal{P}^{\alpha}_{2} ~ \text{and} ~ \mathcal{P}^{\alpha}_{3}$ are subsets of $[0,1]$. Under (A1) and (A2), $\mathcal{P}^{\alpha}_{3} = [0,\pi^{**}(2)), \mathcal{P}^{\alpha}_{2} = [\pi^{**}(2),\pi^{*}(2)), \mathcal{P}^{\alpha}_{1} = [\pi^{*}(2),1]$, where $\pi^{**}$ and $\pi^{*}$ are the belief states at which $H^{\alpha}(2,1)=H^{\alpha}(2,2)$ and $H^{\alpha}(1,1)= H^{\alpha}(1,2)$ respectively. From Theorem \ref{thm:p} and  \eqref{eq:val}, the value function can be written as, 
\begin{align*}
\begin{split}
V(\pi) &= \text{min} \{C(\pi,1), C(\pi,2) + \rho V(\pi) \mathcal{I}(\pi\in \mathcal{P}^{\alpha}_{1}) \\
&\quad + \rho {\underset{a\in \mathcal{A}}{\sum}} V(T^{\pi}(\pi,a))\sigma(\pi,a) \mathcal{I}(\pi\in \mathcal{P}^{\alpha}_{2}) \\
&\quad + \rho V(\pi) \mathcal{I}(\pi\in \mathcal{P}^{\alpha}_{3}) \}
\end{split}
\end{align*}
The explicit dependence of the filter on the belief $\pi$ results in discontinuous value function. The optimal policy in general has multiple thresholds and the stopping region in general is non-convex. 
\subsubsection*{Example}
Fig. \ref{non_cnv} displays  the value function and optimal policy for a toy example having the following parameters:
\[ B = \begin{bmatrix}
0.8 & 0.2 \\
0.3 & 0.7
\end{bmatrix}, 
P =\begin{bmatrix}
1 & 0 \\
0.06 & 0.94
\end{bmatrix},
c = \begin{bmatrix}
1 & 2 \\
2.5 & 0.5
\end{bmatrix}.
\]

The parameters for the market observer are chosen as:
$d = 1.25$, $\textbf{f}=[0 ~ 3]$, $\alpha = 0.8$ and $\rho = 0.9$. 

From Fig. \ref{non_cnv} it is clear that the market observer has a double threshold policy and the value function is discontinuous.
The double threshold policy is unusual from a signal processing point of view. Recall that $\pi(2)$ depicts the posterior probability
of no change. The market observer ``changes its mind" - it switches from no change to change as the posterior
probability of change decreases! Thus the global decision (stop or continue) is a non-monotone function of the posterior
probability obtained from local decisions in the agent based model. The example illustrates the unusual behaviour of the social
learning filter.

\begin{figure}[!t] 
\centering
\includegraphics[scale=0.4]{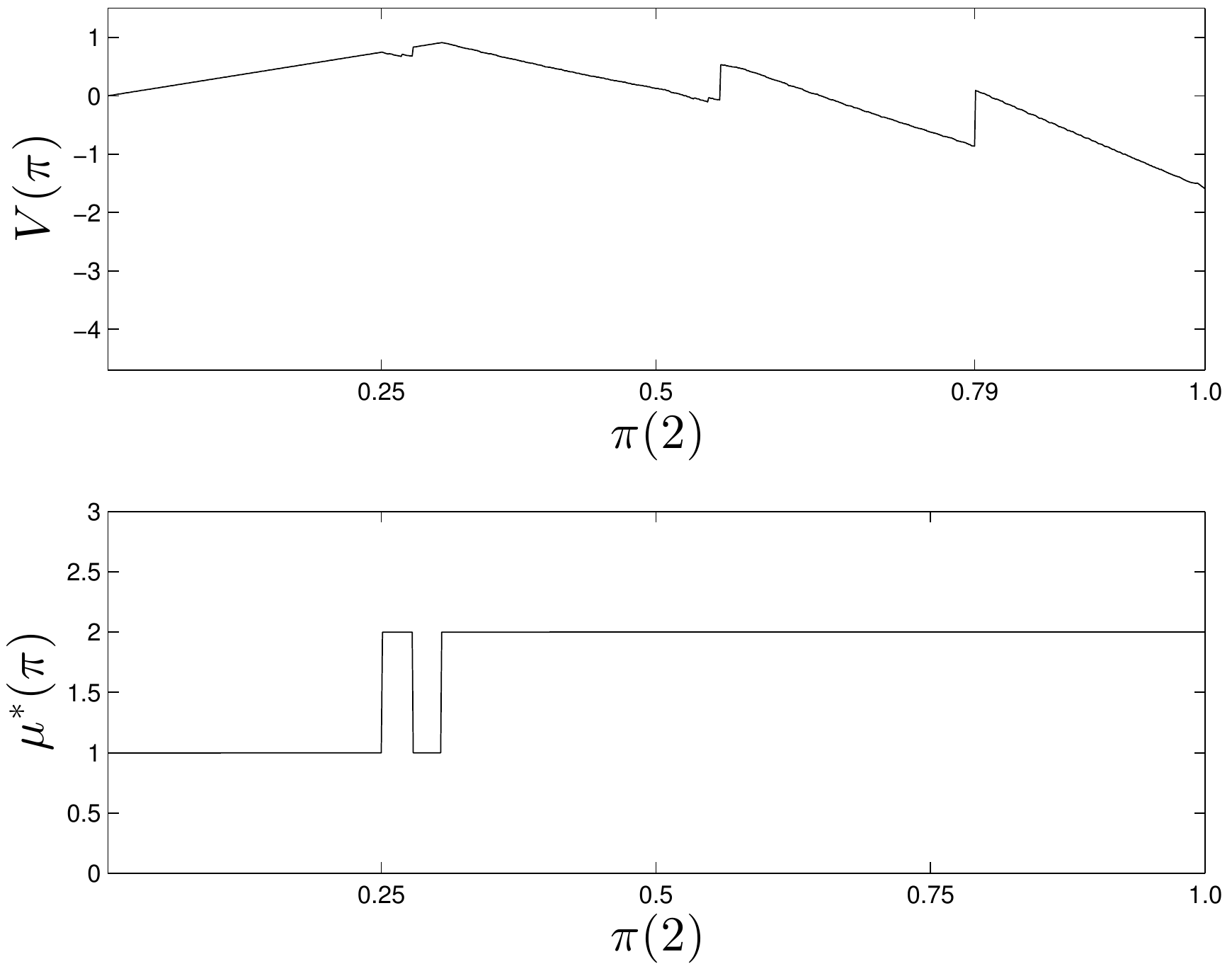}
\caption{The value function $V(\pi)$ and the double threshold optimal policy $\mu^*(\pi)$  are  plotted versus  $\pi(2)$. The significance of the double threshold policy is that the stopping regions are non-convex. The implication of the non-convex stopping set for the market observer is that - if he believes that it is optimal to stop, it need not be optimal to stop when his belief is larger.}
\label{non_cnv}
\end{figure}

\section{Dataset Example} \label{sec:dat}
\subsection{Tech Buzz Game and Model:}
To validate the framework, we consider the data from Tech Buzz Game, a stock market simulation launched by Yahoo! Research and O'Reilly Media to gain insights into forecasting high-tech events and trades. The game consisted of multiple sub-markets trading stocks of contemporary rival technologies. At each trading instant, the traders (or players) had access to the search ``buzz" around each of the stocks. The buzz was an indicator of the number of users scouring Yahoo! search on the stock over the past seven days, as a percentage of all searches in the same market.  Thus, if searches for the stock named ``SKYPE" make up 80 percent of all Yahoo! searches in the telecommunication application software market, then SKYPE's buzz score is 80. The buzz scores of all technologies within a market always add up to 100. The dataset was chosen to demonstrate the framework as the trading information was made available by Yahoo!.

The stock market simulation is modelled as follows. The state $X$ is chosen to represent value of the stock, with $X = 1$ indicating a high valued stock and $X = 2$ indicating a stock of low value. It is well known that if the perceived value is more then it is going to be popular. Hence, the noisy observations are taken to be the buzz scores which are a proxy for the popularity of the stock \cite{CPK08}. For tractability, is assumed that all agents have the same attitude towards risk, i.e, $\alpha$ is the same for all agents. The agents choose to buy ($a=1$) or sell ($a=2$) depending on the cost and the belief updated using the buzz score. On each day the stock is traded, we consider only the agent which buys or sells maximum shares and record its trading decision (positive or negative values in the dataset) as buying or selling a unit of stock. This is reasonable assumption in the sense that the agent trading maximum shares (``big players" in finance) will significantly influence the public belief. 

\subsection{Dataset}
The buzz score for the stocks, \textit{SKYPE} and \textit{IPOD} trading in markets \textit{VOIP} and \textit{PORTMEDIA} respectively, was obtained for the period from April 1, 2005 to July 27, 2005 from the Yahoo! Dataset \footnote{``Yahoo! Webscope", \url{http://research.yahoo.com/Academic_Relations} \\ ydata-yrbuzzgame-transactions-period1-v1.0, ydata-yrbuzzgame-buzzscores-period1-v1.0}. Value of the stock is basically a function of the payout of the stock and its dividend \cite{CPK08}. The payout and dividend are directly proportional to the buzz score. Using the buzz score, value of the stock was calculated with the method suggested in \cite{CPK08}.  Space discretized value of the stocks \textit{SKYPE} and \textit{IPOD} during the period is shown in the Fig. \ref{fig_sim} and Fig. \ref{fig_sim2} along with the scaled buzz score. 
\begin{figure}[!t] 
\centering
\includegraphics[scale=0.4]{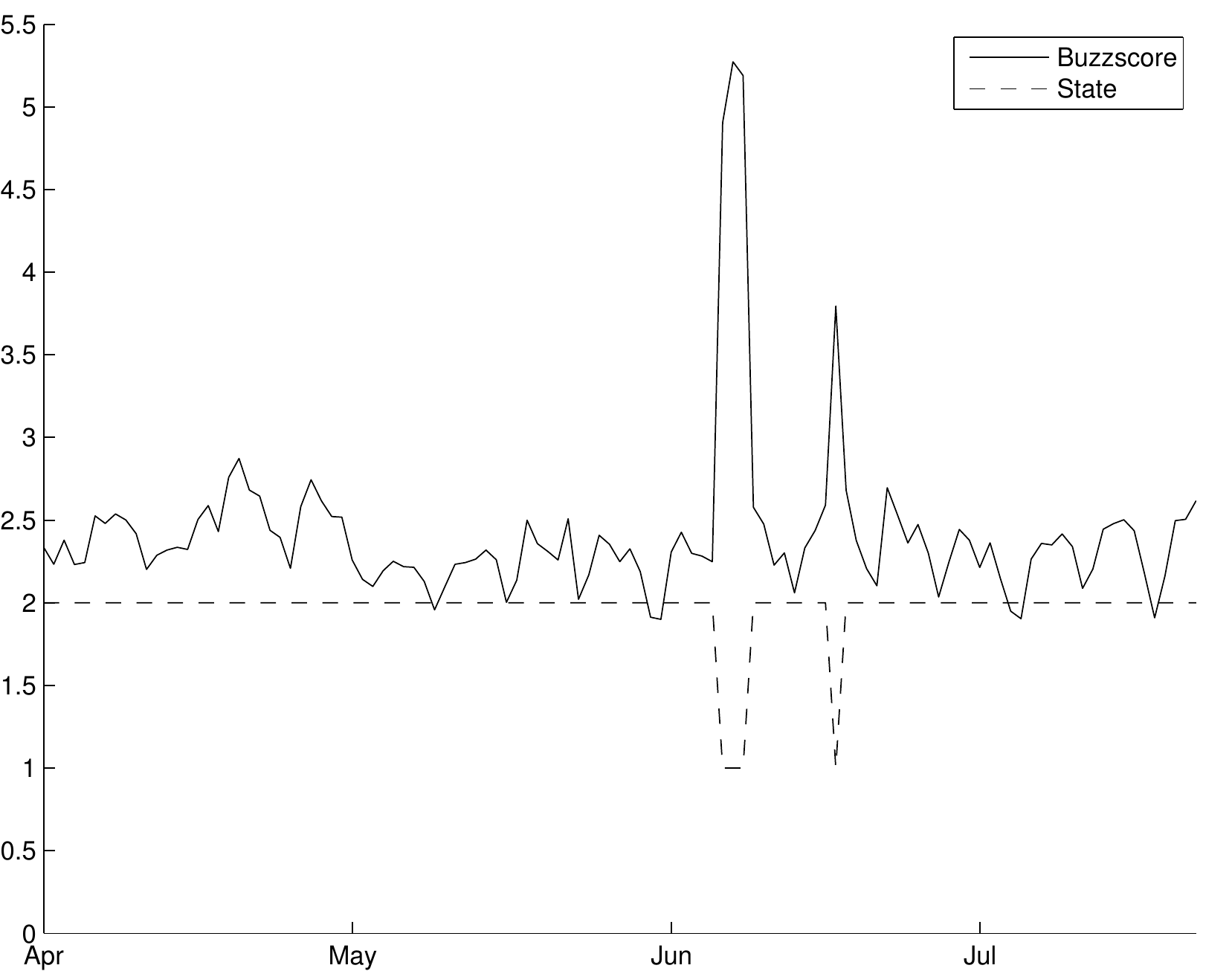}
\caption{Space discretized value of the stock \textit{SKYPE} and the scaled buzz scores during April - July is shown. It is seen that the value changed in the month of June. 
The market observer's aim is to detect the change in the value, i.e, when $x_{k} = 1$, using only the trading decisions. }
\label{fig_sim}
\end{figure}

\subsubsection{Quickest detection for SKYPE}
By eyeballing the data in Fig. \ref{fig_sim}, it is seen that the value changed during the month of June. To apply the quickest detection protocol, we consider a window from May $17$ to June $8$. It is observed that the price was (almost) constant during this period with a value close to \$13 per stock. The trading decisions (along with the value) and the public belief during this period are shown in Fig. \ref{st_p}.
\begin{figure}[!t]
\centering
\includegraphics[scale=0.4]{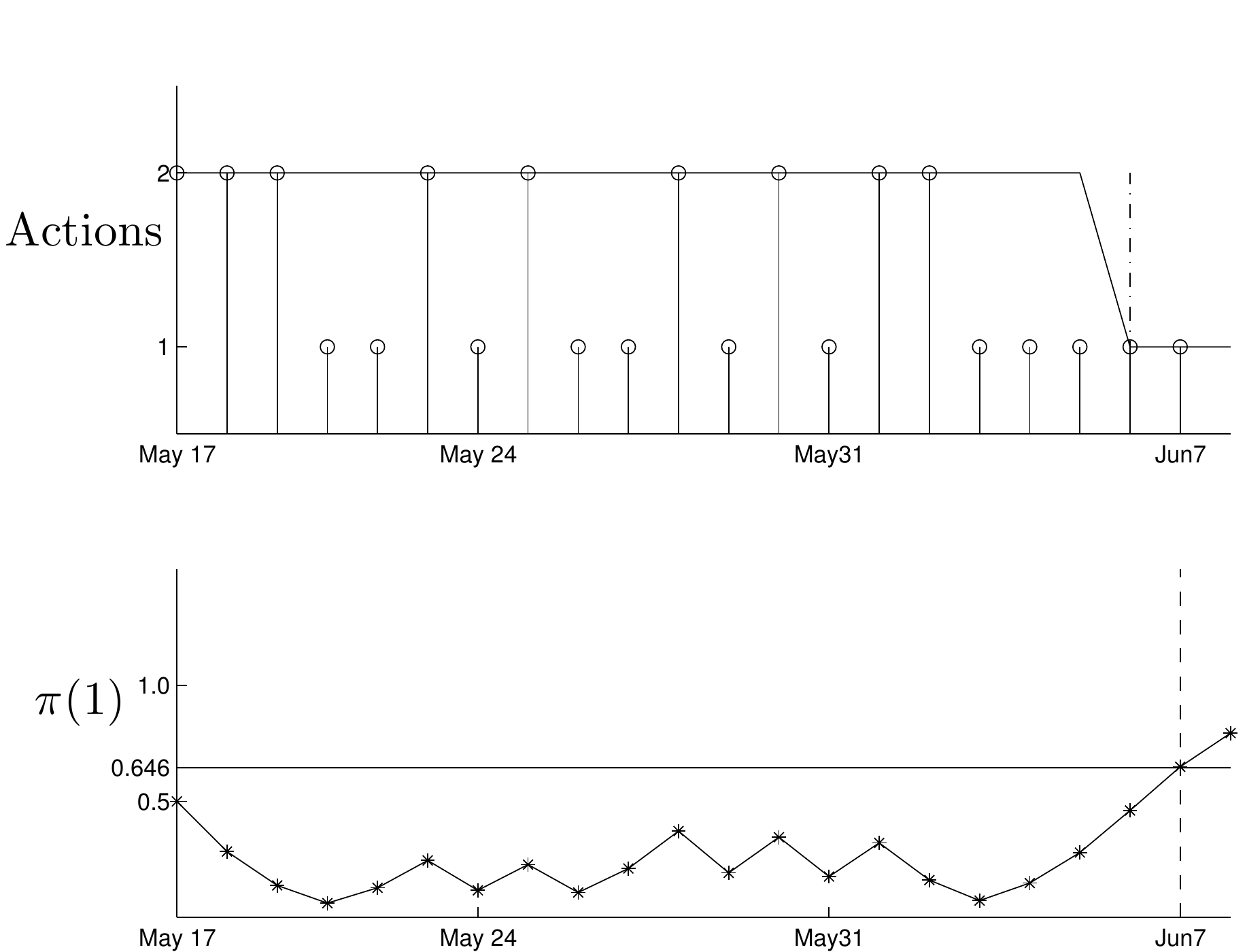}
\caption{The daily trading decisions of the agents is shown with the corresponding belief update. Here $a=1$ corresponds to buying and $a=2$ corresponds to selling the stock. Since $\pi(2) \in [0,0.354]$ is the stopping region, it corresponds to $\pi(1) \geq 0.646.$ The change was detected one day after it occurred.} 
\label{st_p}
\end{figure} 

The local and global costs for the market observer were chosen as:
\[ c = \begin{bmatrix}
0.5 & 1 \\
1 & 0.5
\end{bmatrix},
\textbf{f} = \begin{bmatrix}
0 & 2 
\end{bmatrix},
d = 0.8
\]

The model parameters were chosen as follows: \\
\[ B = \begin{bmatrix}
0.7 & 0.3 \\
0.3 & 0.7
\end{bmatrix},
P = \begin{bmatrix}
1 & 0 \\
0.04 & 0.96
\end{bmatrix}.
\]
The choice of the parameters for the observation matrix $B$ was motivated by the experimental evidence provided in \cite{BEFY14}, that when there is social learning ``alone", the trading rate was $71 \%$ based on peer effects. Since the local decision likelihood matrix $R^{\pi} = B$ in the social learning region (in our model), the parameters were so chosen. Parameters in the transition matrix $P$ were chosen to reflect the time window considered, as $\mathbb{E}\{\tau_{0}\} = 25$.

\begin{figure}[!t]
\centering
\includegraphics[scale=0.45]{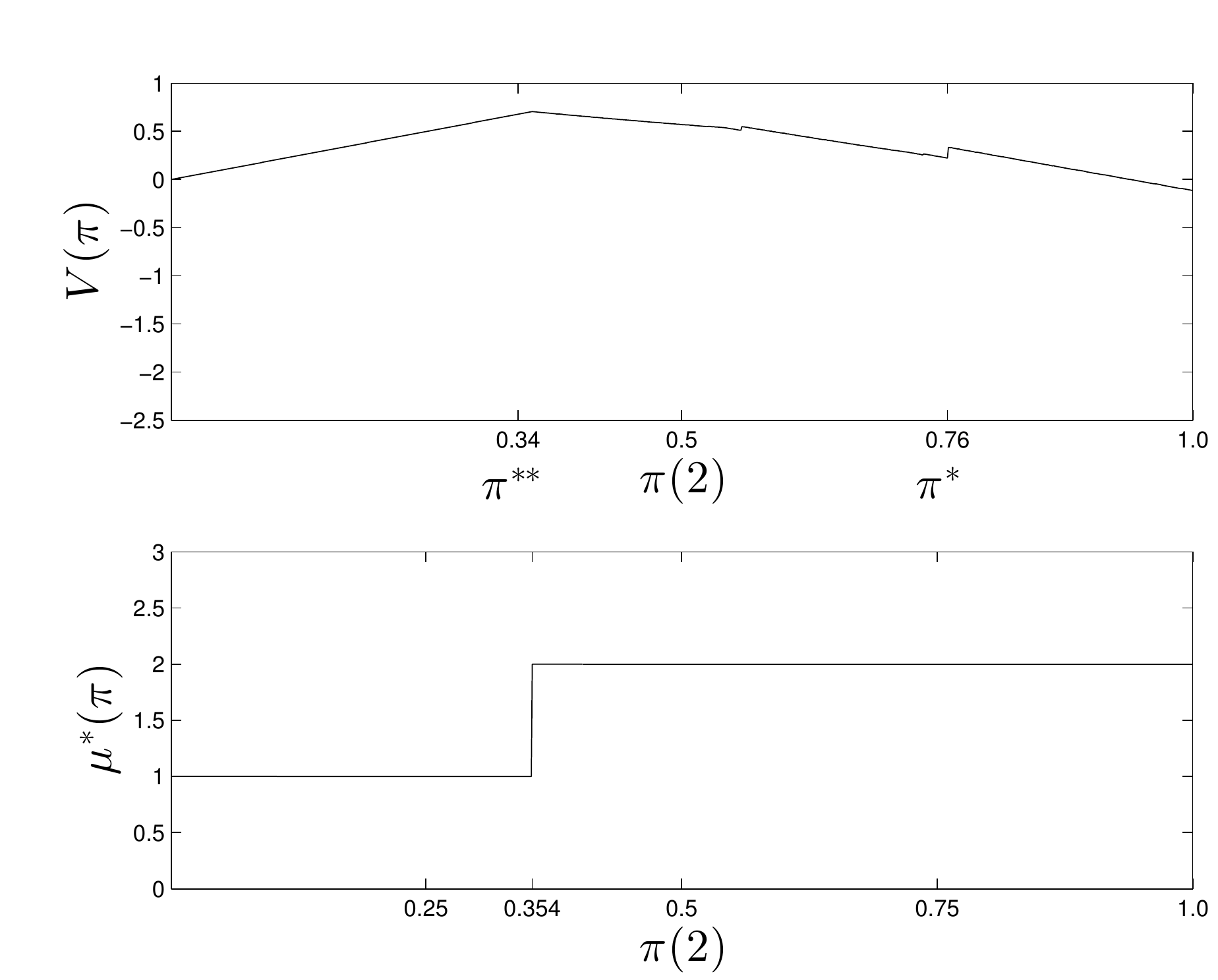}
\caption{Value function and the optimal policy plotted over $\pi(2)$ for $\alpha = 0.45$. Here $\pi^{**}$ and $\pi^{*}$ correspond to the boundary points of the social learning region. $\mu^{*}(\pi) = 1$ corresponds to stop and $\mu^{*}(\pi) = 2$ corresponds to continue. $\pi(2) \in [0,0.354]$ corresponds to the stopping region.}
\label{val_ex}
\end{figure}
It was observed that the state changed on June $6$ and for a risk-aversion factor of $\alpha = 0.45$, it was detected on June $7$.
The value function and the optimal policy for the market observer are shown in Fig. \ref{val_ex}. The stopping set corresponds to $\pi(2) \in [0,0.354]$. The regions $\pi(2) \in [0,0.34)$ and $\pi(2) \in [0.76,1]$ correspond to the regions where social learning is absent. It can be observed that the value function is discontinuous.

\subsubsection{Quickest detection for IPOD}
\begin{figure}[!t] 
\centering
\includegraphics[scale=0.4]{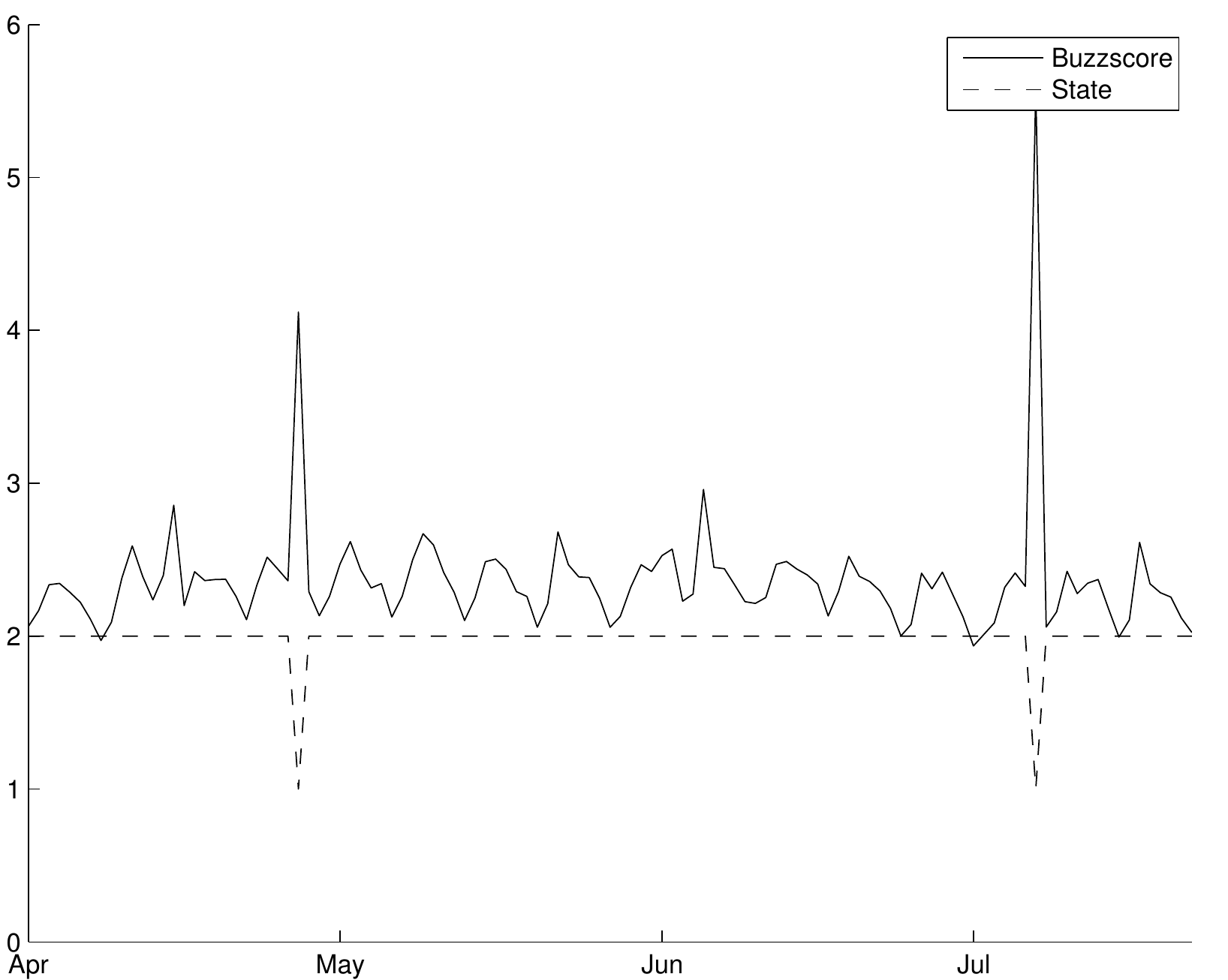}
\caption{Space discretized value of the stock \textit{IPOD} and the scaled buzz scores during April - July is shown. It is seen that the value changed during April and July. The market observer's aim is to detect the change in the value, i.e, when $x_{k} = 1$, using only the trading decisions.}
\label{fig_sim2}
\end{figure}
From Fig. \ref{fig_sim2}, it is seen that the value changed during April and July. To apply the quickest detection protocol, we consider a window from July $2$ to July $10$. It is observed that the price was (almost) constant during this period with a value close to \$17 per stock. The trading decisions (along with the value) and the public belief during this period are shown in Fig. \ref{st_p2}.
\begin{figure}[!t]
\centering
\includegraphics[scale=0.4]{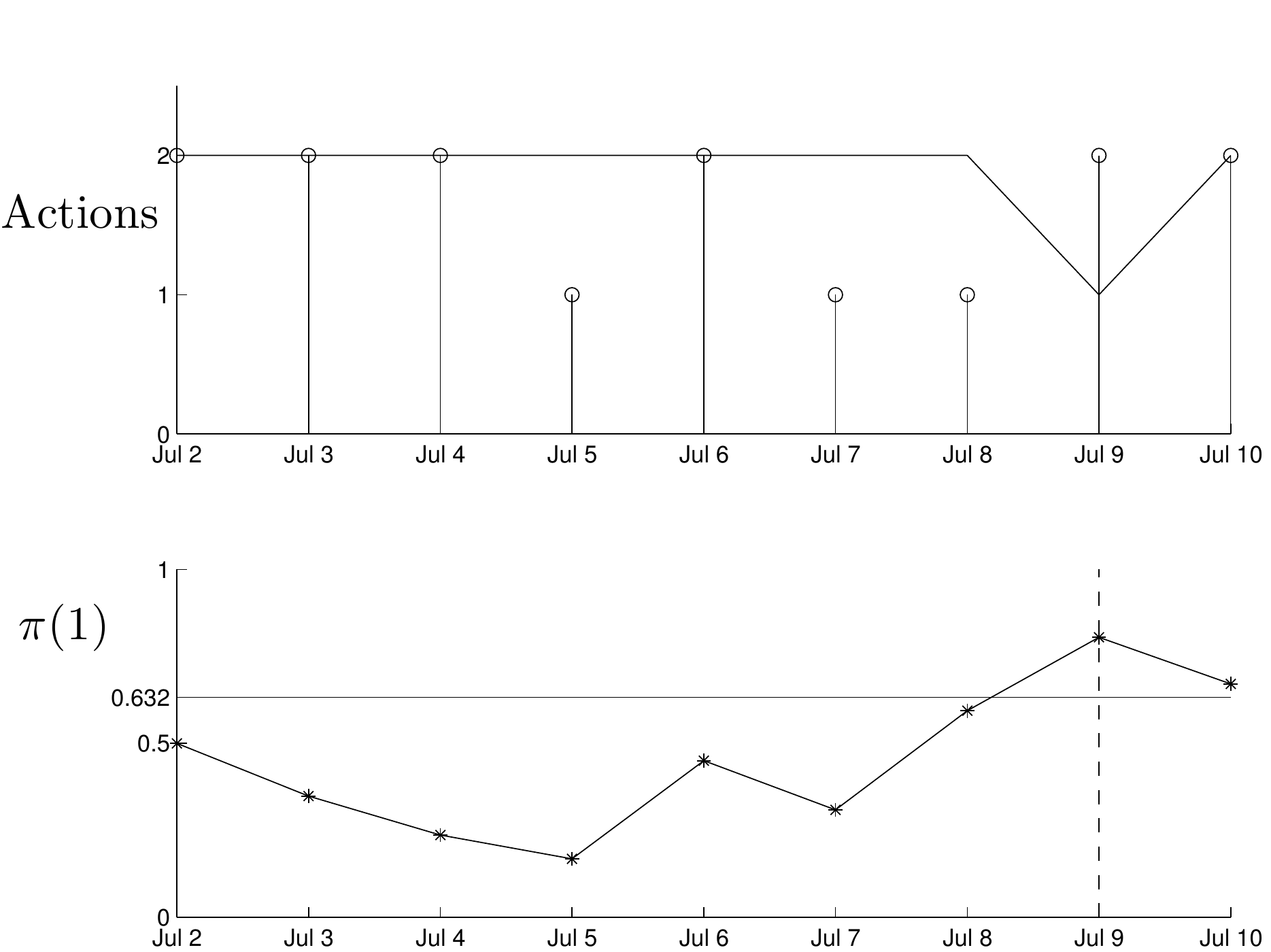}
\caption{The daily trading decisions of the agents is shown with the corresponding belief update. Here $a=1$ corresponds to buying and $a=2$ corresponds to selling the stock. Since $\pi(2) \in [0,0.368]$ is the stopping region, it corresponds to $\pi(1) \geq 0.632.$ The change was detected on the day it occurred with a higher penalty on the delay.} 
\label{st_p2}
\end{figure} 

The local and global costs for the market observer were chosen as:
\[ c = \begin{bmatrix}
0.5 & 1 \\
1 & 0.5
\end{bmatrix},
\textbf{f} = \begin{bmatrix}
0 & 1.8 
\end{bmatrix},
d = 0.95
\]

The model parameters were chosen as follows: \\
\[ B = \begin{bmatrix}
0.7 & 0.3 \\
0.3 & 0.7
\end{bmatrix},
P = \begin{bmatrix}
1 & 0 \\
0.11 & 0.89
\end{bmatrix}.
\]
Parameters in the transition matrix $P$ were chosen to reflect the time window considered, as $\mathbb{E}\{\tau_{0}\} = 9$.

\begin{figure}[!t]
\centering
\includegraphics[scale=0.45]{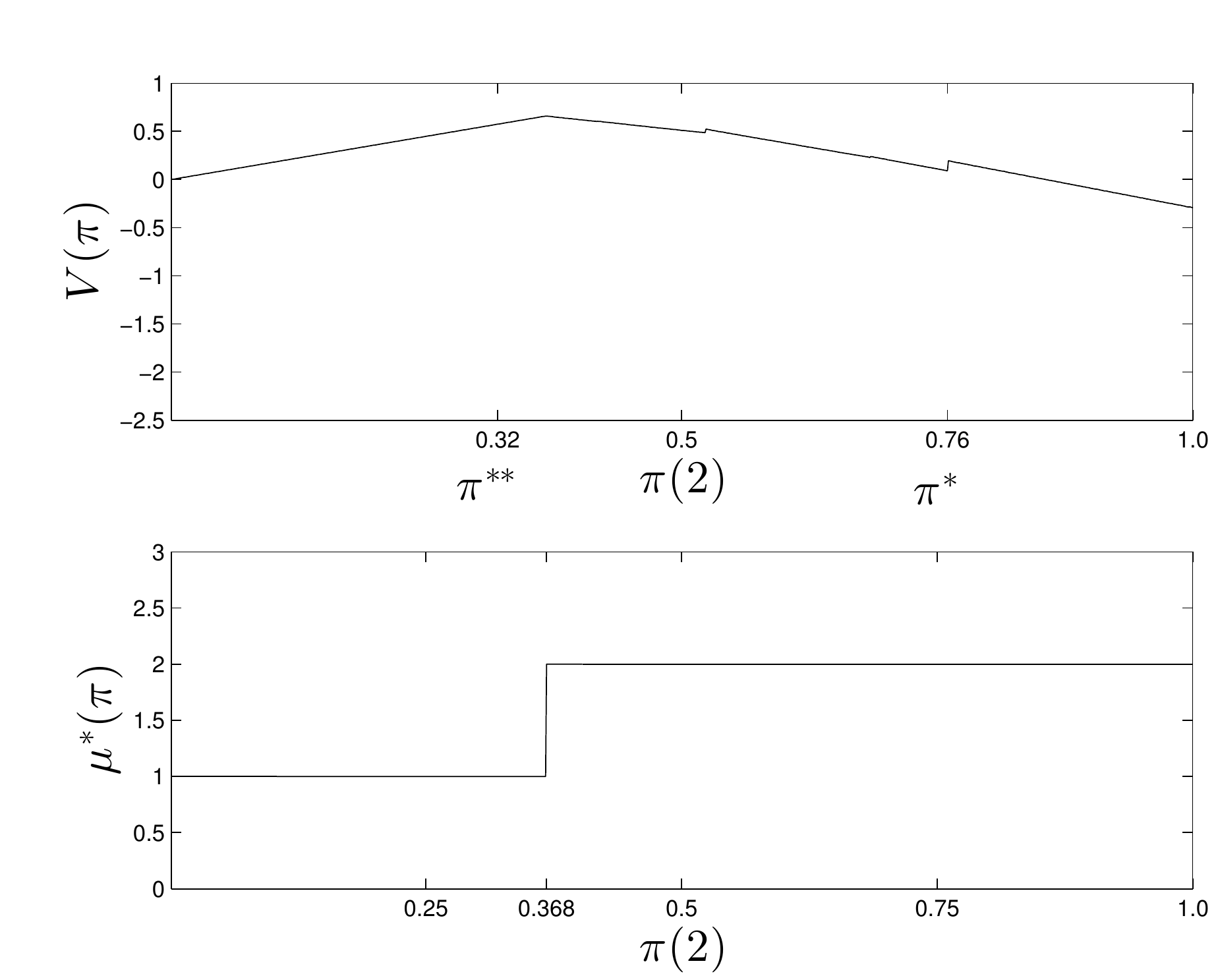}
\caption{Value function and the optimal policy plotted over $\pi(2)$ for $\alpha = 0.45$. Here $\pi^{**}$ and $\pi^{*}$ correspond to the boundary points of the social learning region. $\mu^{*}(\pi) = 1$ corresponds to stop and $\mu^{*}(\pi) = 2$ corresponds to continue. $\pi(2) \in [0,0.368]$ corresponds to the stopping region.}
\label{val_ex2}
\end{figure}
It was observed that the state changed on July $9$ and for a risk-aversion factor of $\alpha = 0.45$, it was detected on July $9$. It is seen that when the delay penalty is increased, the change is detected on the same day. The value function and the optimal policy for the market observer are shown in Fig. \ref{val_ex2}. The stopping set corresponds to $\pi(2) \in [0,0.368]$. 

\section{Conclusion}
The paper provided a Bayesian formulation of the problem of quickest detection of change in the value of a stock using the decisions of socially aware risk averse agents. It highlighted the reasons for this problem to be non-trivial - the stopping region is in general non-convex; and it also accounted for the agents' risk attitude by considering a coherent risk measure, CVaR, instead of the expected value measure in the agents' optimization problem. Results which characterize the structural properties of social learning under the CVaR risk measure were provided and the importance of these results in understanding the global behaviour was discussed. It was observed that the behaviour of these risk-averse agents is, as expected, different from risk neutral agents. Risk averse agents herd sooner and don't prefer to ``learn" from the crowd, i.e, social learning region is smaller the more risk-averse the agents are.  Finally, the model was validated using a financial dataset from Yahoo! Tech Buzz game.
From a signal processing point of view, the formulation and solutions are non-standard due to the three properties described in Section \ref{sec:intro}.

In current work, we are interested in determining structural results for the  optimal change detection. Structural results for POMDPs were developed
in \cite{KD07,KP14} and it is worthwhile extending these results to the current framework. When the market observer incurs a measurement cost,
quickest change detection often has a monotone policy as described in \cite{Kri13}. It is of interest to generalize these results to the current setup.


%

\appendices
\section{Preliminaries and Definitions:}
\begin{definition}{\textit{MLR Ordering \cite{MS02}} ($\geq_{r}$):} Let $\pi_{1},\pi_{2}\in \Pi(X)$ be any two belief state vectors. Then $\pi_{1}\geq_{r}\pi_{2}$ if 
\begin{equation*} \label{def:mlrdef}
\pi_{1}(i)\pi_{2}(j) \leq \pi_{2}(i)\pi_{1}(j), ~ i<j, i,j \in \{1,\hdots,X\}.
\end{equation*}
\end{definition}
\begin{definition}{\textit{First-Order Stochastic Dominance} ($\geq_{s}$):} Let $\pi_{1},\pi_{2}\in \Pi(X)$ be any two belief state vectors. Then $\pi_{1}\geq_{s}\pi_{2}$ if 
\begin{equation*}
{\overset{X}{\underset{i=j}{\sum}}} \pi_{1}(i) \geq {\overset{X}{\underset{i=j}{\sum}}} \pi_{2}(i) ~ \text{for} ~ j \in \{1,\hdots,X\}.
\end{equation*}  
\end{definition}

\begin{lemma}{\cite{MS02}} \label{lem:dec}
$\pi_{2} \geq_{s} \pi_{1}$ iff for all $v\in \mathcal{V}$, $v'\pi_{2}\leq v'\pi_{1}$, where $\mathcal{V}$ denotes the space of $X$- dimensional vectors $v$, with non-increasing components, i.e, $v_{1} \geq v_{2} \geq \hdots v_{X}$.
\end{lemma}
\begin{lemma}{\cite{MS02}} \label{lem:inc}
$\pi_{2} \geq_{s} \pi_{1}$ iff for all $v\in \mathcal{V}$, $v'\pi_{2}\geq v'\pi_{1}$, where $\mathcal{V}$ denotes the space of $X$- dimensional vectors $v$, with non-decreasing components, i.e, $v_{1} \leq v_{2} \leq \hdots v_{X}$.
\end{lemma}

Let $\Pi(X) {\overset{\Delta}{=}} \{ \pi\in\mathbb{R}^{X} : \textbf{1}_{X}'\pi = 1, 0\leq \pi(i) \leq 1 ~ \text{for all} ~ i\in \mathcal{X} \}$
\begin{definition}{\textit{Submodular function} \cite{Top98}:} A function $f : \Pi(X)\times \{1,2\} \rightarrow \mathbb{R}$ is submodular if $f(\pi,u)-f(\pi,\bar{u}) \leq f(\bar{\pi},u)-f(\bar{\pi},\bar{u})$, for $\bar{u}\leq u, \pi \geq_{r} \bar{\pi}$. 
\end{definition}
\begin{definition}{\textit{Single Crossing Condition} \cite{Top98}:} A function $g : \mathcal{Y} \times \mathcal{A} \rightarrow \mathbb{R}$ satisfies a single crossing condition in $(y,a)$ if 
\begin{equation*}
g(y,a) - g(y,\bar{a}) \geq 0 \Rightarrow g(\bar{y},a) - g(\bar{y},\bar{a}) \geq 0
\end{equation*}
for $\bar{a}>a$ and $\bar{y}>y$. For any such function $g$ , 
\begin{equation} \label{eq:sgc}
a^{*}(y) = {\underset{a}{\text{argmin}}} ~ g(y,a) ~ \text{is increasing in} ~ y.
\end{equation}
\end{definition}
\begin{theorem}{\cite{Top98}} \label{thm:sb}
If $f : \Pi(X) \times \{1,2\} \rightarrow \mathbb{R}$ is sub-modular, then there exists a $u^{*}(\pi) = {\underset{u\in \{1,2\}}{\text{argmin}}} f(\pi,u)$ satisfying,
\begin{equation*}
\bar{\pi} \geq_{r} \pi \Rightarrow u^{*}(\pi) \leq u^{*}(\bar{\pi})
\end{equation*}
\end{theorem}

\section{Proofs}
The following lemmas are required to prove Theorem \ref{thm:a} and Theorem \ref{thm:p}. The results will be proved for general state and observation spaces having two actions. 
\begin{lemma} \label{lem:z}
For a finite state and observation alphabet, $\armz \{  z + \frac{1}{\zeta} \mathbb{E}_{y}[{\max} \{ (c(x,a)-z),0 \}]  \}$ is equal to $c(i,a)$ for some $i\in\{1,2,\hdots,X\}$.
\end{lemma}

\begin{proof}
Let $\eta_{y}$ be the belief update (p.m.f) with observation $y$, i.e, $\eta_{y}(i) = \Prby(x=i)$. Let $F_{y}(x)$ denote the cumulative distribution function. For simplicity of notation, let $h_{y}(z) = z+\frac{1}{\alpha}\E_{y}[\ma]$. The extremum of $h_{y}(z)$ is attained where the derivative is zero. It is obtained as follows.

\resizebox{1.0\linewidth}{!} {
  \begin{minipage}{\linewidth}
\begin{align*}
h_{y}(z) &= z+\frac{1}{\alpha}\E_{y}[\ma] \\
h_{y}'(z) &= 1 + \frac{1}{\alpha} \lia \frac{\E_{y}[\max\{c(x,a) - z -\Delta z,0\}] - \E_{y}[\ma]}{\Delta z} \\
  &= 1 + \frac{1}{\alpha} \E_{y} \left( \lia \frac{\max\{c(x,a) - z -\Delta z,0\} - \ma}{\Delta z}\right) \\
  &= 1 + \frac{1}{\alpha} \E_{y} \left( 0 \times \mathcal{I}_{0 > (c(x,a) - z)} - 1 \times \mathcal{I}_{(c(x,a) - z) > 0} \right) \\
  &= 1 - \frac{1}{\alpha} \Prby(c(x,a)>z).
\end{align*}
\end{minipage} }

Also, $h_{y}''(z) = \frac{1}{\alpha} \dfrac{d}{dz}(F_{y}(z))$ and therefore $h_{y}''(z) \geq 0$. We have, $\armz~ \{ h_{y}(z) \} = \{ z : \Prby(c(x,a)>z) = \alpha \}$. Since $X$ is a random variable, $c(x,a)$ is a random variable with realizations $c(i,a)$ for $i \in \{ 1,\hdots X \}$. Hence $z = c(i,a)$ for some $i\in\{1,2,\hdots,X\}$.
\end{proof}
The result of Lemma \ref{lem:z} is similar to Proposition $8$ in \cite{RU02}. It was shown in \cite{Lov87} that $\eta_{y+1} \geq_{r} \eta_{y}$. Also, MLR dominance implies first order dominance, i.e, $\eta_{y+1} \geq_{s} \eta_{y}$. 

\begin{lemma} \label{lem:kl}
 Let $l$ be the index such that $\armz \{  h_{y}(z)  \} = c(l,a)$ and $k$ be the index such that $\armz \{  h_{y+1}(z)  \} = c(k,a)$. For all $y\in \{ 1,2\hdots,Y \}$,  $k \geq l$. 
\end{lemma}

\begin{proof}
Proof is by contradiction.  From lemma \eqref{lem:z}, we have $F_{y}(c(l,a)) = 1-\alpha$ and $F_{y+1}(c(k,a)) = 1-\alpha$. Suppose $l>k$. We know that $F_{y+1}(z)$ is a monotone function in $z$. Since $l>k$, $F_{y+1}(c(l,a)) > 1-\alpha$. But, by definition of first order stochastic dominance, $F_{y}(z) \geq F_{y+1}(z)$ for all $z$. Therefore, $F_{y}(c(l,a))\geq F_{y+1}(c(l,a)) > 1-\alpha$, a contradiction.
\end{proof}
From lemma \ref{lem:z} and equation \eqref{eq:hya}, we have 
\begin{align*}
H^{\alpha}(y,2) = c(l,2) + \frac{1}{\alpha} \sul \eta_{y}(i) (c(i,2)-c(l,2)), \\
H^{\alpha}(y+1,2) = c(k,2) + \frac{1}{\alpha} \sk \eta_{y+1}(i) (c(i,2)-c(k,2)) \\
\end{align*}
\begin{lemma} \label{lem:hy1}
$H^{\alpha}(y,2) \geq H^{\alpha}(y+1,2)$ if $\alpha \geq 1 - \Prby(x=X)$.
\end{lemma}

\begin{proof}
From the definitions of $H^{\alpha}(y,2)$ and $H^{\alpha}(y+1,2)$ we have,
\begin{align} \label{eq:h1}
H^{\alpha}(y,2) - H^{\alpha}(y+1,2) &= c(l,2) - c(k,2) + \nonumber \\ 
 \frac{1}{\alpha} \sul \eta_{y}(i) (c(i,2)-c(l,2)) + 
 &\frac{1}{\alpha} \sk \eta_{y+1}(i) (c(k,2)-c(i,2))  \nonumber \\
 &\geq c(l,2) - c(k,2) +  \nonumber \\
  \frac{1}{\alpha} \sul \eta_{y}(i) (c(i,2)-c(l,2)) + 
  &\frac{1}{\alpha} \sk \eta_{y}(i) (c(k,2)-c(i,2)) 
\end{align}
Equation \eqref{eq:h1} follows from lemma \ref{lem:dec} and can be simplified as
\begin{align*} 
H^{\alpha}(y,2) - H^{\alpha}(y+1,2) &\geq c(l,2) - c(k,2) + \nonumber \\
 \frac{1}{\alpha} \sul \eta_{y}(i) (c(k,2)-c(l,2)) +  
&\frac{1}{\alpha} \skl \eta_{y}(i) (c(k,2)-c(i,2)) \nonumber \\
&\geq c(l,2) - c(k,2) - \frac{1}{\alpha} \Gamma' \eta_{y}
\end{align*}
where $\Gamma$ is such that $\Gamma_{i} = c(l,2) - c(k,2)$ for $i=1,\hdots,l-1$ and $\Gamma_{i} = c(i,2) - c(k,2)$ for $i=l,\hdots k-1$. Clearly, $\Gamma_{i}\geq 0$ and decreasing. Right hand side of inequality attains its maximum when $k=X$ and $l=1$ and $\Gamma_{i} = c(l,2) - c(k,2)$ for all $i$. Therefore, we have
\begin{align}
H^{\alpha}(y,2) - H^{\alpha}(y+1,2) &\geq  c(l,2) - c(k,2) - \frac{1}{\alpha} \Gamma' \eta_{y} \nonumber \\
&\hspace{-3.5cm} \geq (c(l,2) - c(k,2)) - \frac{1}{\alpha} (c(l,2) - c(k,2)) (1-\Prby(x=X)) \nonumber 
\end{align}
After rearrangement we have,
\begin{multline*}
H^{\alpha}(y,2) - H^{\alpha}(y+1,2) \\ \geq  \frac{\alpha - (1-\Prby(x=X))}{\alpha} (c(l,2) - c(k,2))  \nonumber 
\end{multline*}
Since $\alpha \geq 1 - \Prby(x=X)$ and $(c(l,2) - c(k,2))\geq0$ (follows from lemma \ref{lem:kl} and assumption (A2)), we have $H^{\alpha}(y,2) \geq H^{\alpha}(y+1,2)$.
\end{proof}

\vspace{-0.1cm}
From Lemma \ref{lem:z} and  \eqref{eq:hya}, we have 
\begin{align*}
H^{\alpha}(y,1) = c(l,1) + \frac{1}{\alpha} \slx \eta_{y}(i) (c(i,1)-c(l,1)), \\
H^{\alpha}(y+1,1) = c(k,1) + \frac{1}{\alpha} \skx \eta_{y+1}(i) (c(i,1)-c(k,1)) \\
\end{align*}
\begin{lemma} \label{lem:hy2}
$H^{\alpha}(y+1,1) \geq H^{\alpha}(y,1)$ if $\alpha \geq 1 - \Prbyy(x=X)$.
\end{lemma}

\begin{proof}
From the definitions of $H^{\alpha}(y+1,1)$ and $H^{\alpha}(y,1)$ we have,
\begin{align} \label{eq:ha2}
H^{\alpha}(y+1,1) - H^{\alpha}(y,1) &= c(k,1) - c(l,1) + \nonumber \\ 
\hspace{0cm}\frac{1}{\alpha} \skx \eta_{y+1}(i) (c(i,1)-c(k,1))   
& - \frac{1}{\alpha} \slx \eta_{y}(i) (c(i,1)-c(l,1))  \nonumber \\
 &\hspace{-4cm}\geq c(k,1) - c(l,1) 
+  \frac{1}{\alpha} \skx \eta_{y+1}(i) (c(i,1)-c(k,1))
\nonumber \\ & \hspace{-2cm}\ - \frac{1}{\alpha} \slx \eta_{y+1}(i) (c(i,1)-c(l,1))
\end{align}
 Equation \eqref{eq:ha2} follows from lemma \ref{lem:inc} and can be simplified as
\begin{align*}
H^{\alpha}(y+1,1) - H^{\alpha}(y,1) &\geq c(k,1) - c(l,1)  \\ 
& +  \frac{1}{\alpha} \skx \eta_{y+1}(i) (c(i,1)-c(k,1))  \\
&  - \frac{1}{\alpha} \slx \eta_{y+1}(i) (c(i,1)-c(l,1))\\
& \geq c(k,1) - c(l,1) - \frac{1}{\alpha} \Delta' \eta_{y+1}
\end{align*}
where $\Delta$ is such that $\Delta_{i} = c(i,1) - c(l,1)$ for $i=l,\hdots,k$ and $\Delta_{i} = c(k,1) - c(l,1)$ for $i=k+1,\hdots X$. Clearly, $\Delta_{i}\geq 0$ and decreasing. Right hand side of inequality  attains its maximum when $k=X$ and $l=1$ and $\Delta_{i} = c(k,1) - c(l,1)$ for all $i$. Therefore, we have
\begin{multline*}
 H^{\alpha}(y+1,1) - H^{\alpha}(y,1) \\ \geq (c(k,1) - c(l,1)) 
 - \frac{1}{\alpha} (c(k,1) - c(l,1)) (1-\Prbyy(x=X)) \nonumber 
\end{multline*}
After rearrangement we have,
\begin{multline}
H^{\alpha}(y+1,1) - H^{\alpha}(y,1)  \\ \geq \frac{\alpha - (1-\Prbyy(x=X))}{\alpha} (c(k,1) - c(l,1))  
\end{multline}
Since $\alpha \geq 1 - \Prbyy(x=X)$ and $c(k,1) - c(l,1)\geq0$ (follows from lemma \ref{lem:kl} and assumption (A2)), we have $H^{\alpha}(y+1,1) \geq H^{\alpha}(y,1)$.
\end{proof}

\begin{lemma} \label{lem:main}
Let $\alpha \geq (1-\Prby(x=X))$. The function $H^{\alpha}(y,a)$ satisfies the single crossing condition i.e, 
\begin{equation*}
(H^{\alpha}(y,1) - H^{\alpha}(y,2)) \geq 0 \Rightarrow (H^{\alpha}(y+1,1) - H^{\alpha}(y+1,2)) \geq 0
\end{equation*}
\end{lemma}

\begin{proof}
Assume $(H^{\alpha}(y,1) - H^{\alpha}(y,2)) \geq 0$. We have,
\begin{align} \label{thm:1}
H^{\alpha}(y,1) - H^{\alpha}(y,2) &\geq 0 \nonumber \\
\Rightarrow H^{\alpha}(y,1) - H^{\alpha}(y+1,2) &\geq 0
\end{align}
Equation \eqref{thm:1} follows from lemma \ref{lem:hy1}. Also,
\begin{align} \label{thm:2}
H^{\alpha}(y,1) - H^{\alpha}(y+1,2) &\geq 0 \nonumber \\
\Rightarrow H^{\alpha}(y+1,1) - H^{\alpha}(y+1,2) &\geq 0
\end{align}
Equation \eqref{thm:2} follows from Lemma \ref{lem:hy2}. 
\end{proof}

Lemma \ref{lem:main} is a crucial result which helps us to prove Theorem \ref{thm:a} and Theorem \ref{thm:p}.

\textit{Proof of Theorem} \ref{thm:a}: From lemma \ref{lem:main}, $H^{\alpha}(y,a)$ satisfies the single crossing condition and hence is sub-modular in $(y,a)$. Using Theorem \ref{thm:sb}, we get $a^{*}(\pi,y) = {\text{argmin}} H^{\alpha}(y,a)$ is increasing in $y$.

\textit{Proof of Theorem} \ref{thm:p}: From lemma \ref{lem:main}, $H^{\alpha}(y,a)$ satisfies the single crossing condition. It is easily verified that the belief states satisfy the following property
\begin{multline} \label{eq:Hyl}
\{\pi : H^{\alpha}(y,1)-H^{\alpha}(y,2) \geq 0 \} \\ \subseteq 
 \{\pi : H^{\alpha}(y+1,1)-H^{\alpha}(y+1,2) \geq 0 \}
\end{multline}
Equation \eqref{eq:Hyl} says that the curves $\{\pi : H^{\alpha}(y,1)-H^{\alpha}(y,2) = 0 \} ~ \text{for all}~y\in{\mathcal{Y}}$ do not intersect. Also from  \eqref{eq:rpi} and \eqref{eq:Hyl}, it is easily verified that there are at most $Y+1$ local decision likelihood matrices $R^{\pi}$ (can be less than $Y+1$ when $H^{\alpha}(\bar{y},1)-H^{\alpha}(\bar{y},2) > 0 ~ \text{for some~} \bar{y} \in \mathcal{Y}, ~ \text{for all} ~ \pi$). The matrices $R^{\pi}$ from \eqref{eq:rpi} and Theorem \ref{thm:a} are constant on each of the $Y+1$ polytopes.





\ifCLASSOPTIONcaptionsoff
  \newpage
\fi



\bibliographystyle{IEEEtran}
\bibliography{$HOME/styles/bib/vkm}
\end{document}